\RequirePackage{fix-cm}
\documentclass[smallextended,envcountsect]{svjour3}       
\smartqed  
\makeatletter
\long\def\@makecaption#1#2{%
 \captionstyle
 \ifx\@captype\fig@type
   \vskip\figcapgap
 \fi
 \setbox\@tempboxa\hbox{{\floatlegendstyle #1\floatcounterend}%
 \capstrut #2}%
 \ifdim \wd\@tempboxa >\hsize
   {\floatlegendstyle #1\floatcounterend}\capstrut #2\par
 \else
   \hfill\unhbox\@tempboxa\hfill\hfill%
 \fi
 \ifx\@captype\fig@type\else
   \vskip\tabcapgap
 \fi}
\makeatother

\usepackage{graphicx}
\usepackage{amsmath,amsfonts,latexsym,amssymb,xspace,mathrsfs,epsfig,bm,enumitem}
\usepackage[round, authoryear]{natbib}
\usepackage[retainorgcmds]{IEEEtrantools}
\usepackage[bottom]{footmisc}
\usepackage{todonotes}
\RequirePackage[colorlinks,citecolor=blue,urlcolor=blue]{hyperref}

\spnewtheorem{counterexample}[subsection]{Counter-example}{\bf}{\rm}
\spnewtheorem{condition}[subsection]{Condition}{\bf}{\rm}

\newcommand{\cip}{\,\perp\!\!\!\perp}
 \newcommand{\cS}{{\cal S}}  
\newcommand{\E}{\mathbb{E}} 
\newcommand{\pr}{P}
\newcommand{\indo}[2]{\mbox{$#1 \,\cip\, #2$}}
\newcommand{\ind}[3]{\mbox{$#1 \, \cip\, #2 \mid #3$}}
\newcommand{\nind}[3]{\mbox{$#1 \,  \not\!\!\cip\, #2 \mid #3$}}

\newcommand{\ninda}[4]{\mbox{$#1 \not\!\!\cip_{\mbox{\scriptsize $#4$}}\, #2 \mid #3$}} 
\newcommand{\etc}{{\em etc.\/}\xspace}
\newcommand{\viz}{{\em viz.\/}\xspace}
\newcommand{\eg}{{\em e.g.\/}\xspace} \newcommand{\ie}{{\em i.e.\/}\xspace} \newcommand{\as}[1]{\mbox{\rm a.s.\ [$#1$]}}
\newcommand{\aso}{\mbox{\rm almost surely}}  \newcommand{\cd}{\,|\,}

\newcommand{\secref}[1]{\mbox{\S~\ref{sec:#1}}}

\newcommand{\figref}[1]{\mbox{Figure~\ref{fig:#1}}}
\newcommand{\lemref}[1]{\mbox{Lemma~\ref{lem:#1}}}
\newcommand{\tabref}[1]{\mbox{Table~\ref{tab:#1}}}
\newcommand{\defref}[1]{\mbox{Definition~\ref{def:#1}}}

\newcommand{\countref}[1]{\mbox{Counter-example~\ref{count:#1}}}
\newcommand{\Countref}[1]{\mbox{Counter-example~\ref{count:#1}}}
\newcommand{\itref}[1]{\mbox{\ref{it:#1}}}
\newcommand{\theoremref}[1]{\mbox{Theorem~\ref{thm:#1}}}
\newcommand{\thmref}[1]{\mbox{Theorem~\ref{thm:#1}}}

\renewcommand{\eqref}[1]{\mbox{(\ref{eq:#1})}}

\journalname{Statistics in Biosciences}
\begin{document}
  
\title{A Formal Treatment of Sequential
  Ignorability}


\titlerunning{A Formal Treatment of Sequential
  Ignorability} 

\author{A. Philip Dawid \and Panayiota Constantinou}

\authorrunning{A. P. Dawid and P.
  Constantinou} 

\institute{A. Philip Dawid \at
  Statistical Laboratory, DPMMS \\
  Centre for Mathematical Sciences\\
  University of Cambridge\\
  Wilberforce Road\\
  Cambridge CB3 0WB UK\\
  Tel.: +44 (0)1223 766535 \\
  Fax: +44 (0)1223 337956\\
  \email{apd@statslab.cam.ac.uk} 
  \emph{Present address:} of F. Author 
  \and Panayiota Constantinou \at
  Statistical Laboratory, DPMMS \\
  Centre for Mathematical Sciences\\
  University of Cambridge\\
  Wilberforce Road\\
  Cambridge CB3 0WB UK\\
  \email{pc393@cam.ac.uk} }

\date{Received: date / Accepted: date}

\maketitle

\begin{abstract}
  Taking a rigorous formal approach, we consider sequential decision
  problems involving observable variables, unobservable variables, and
  action variables.  We can typically assume the property of {\em
    extended stability\/}, which allows identification (by means of
  ``$G$-computation'') of the consequence of a specified treatment
  strategy if the ``unobserved'' variables are, in fact,
  observed---but not generally otherwise.  However, under certain
  additional special conditions we can infer {\em simple stability\/}
  (or {\em sequential ignorability\/}), which supports $G$-computation
  based on the observed variables alone.  One such additional
  condition is {\em sequential randomization\/}, where the unobserved
  variables essentially behave as random noise in their effects on the
  actions.  Another is {\em sequential irrelevance\/}, where the
  unobserved variables do not influence future observed variables.  In
  the latter case, to deduce sequential ignorability in full
  generality requires additional positivity conditions.  We show here
  that these positivity conditions are not required when all variables
  are discrete.

\end{abstract}

\section{Introduction}
\label{sec:Intro}

We are often concerned with controlling some variable of interest
through a sequence of consecutive actions.  An example in a medical
context is maintaining a critical variable, such as blood pressure,
within an appropriate risk-free range.  To achieve such control, the
doctor will administer treatments over a number of stages, taking into
account, at each stage, a record of the patient's history, that
provides him with information on the level of the critical variable,
and possibly other related measurements, as well as the patient's
reactions to the treatments applied in preceding stages.  Consider,
for instance, practices followed after events such as stroke,
pulmonary embolism or deep vein thrombosis \citep{henderson,Sterne}.
The aim of such practices is to keep the patient's prothrombin time
(international normalized ratio, INR) within a recommended range.
Such efforts are not confined to a single decision and instant
allocation of treatment, marking the end of medical care.  Rather,
they are effected over a period of time, with actions being decided
and applied at various stages within this period, based on information
available at each stage.  So the patient's INR and related factors
will be recorded throughout this period, along with previous actions
taken, and at each stage all the information so far recorded, as well,
possibly, as other, unrecorded information, will form the basis upon
which the doctor will decide on allocation of the subsequent
treatment.

A well-specified algorithm that takes as input the recorded history of
a patient at each stage and gives as output the choice of the next
treatment to be allocated constitutes a \emph{dynamic decision
  strategy}.  Such a strategy gives guidance to the doctor on how to
take into account the earlier history of the patient, including
reactions to previous treatments, in allocating the next treatment.
There can be an enormous number of such strategies, having differing
impacts on the variable of interest.  We should like to have criteria
to evaluate these strategies, and so allow us to choose the one that
is optimal for our problem \citep{Murphy}.

In this paper we develop and extend the decision-theoretic approach to
this problem described by \citet{dawiddidelez}.  A problem that
complicates the evaluation of a strategy is that the data we possess
were typically not generated by applying that strategy, but arose
instead from an observational study.  We thus seek conditions, which
we shall express in decision-theoretic terms, under which we can
identify the components we need to evaluate a strategy from such data.
When appropriate conditions are satisfied, the {\em $G$-computation
  algorithm\/} introduced by \citet{Robins1986,Robins1992} allows us
to evaluate a strategy on the basis of observational data.  Our
decision-theoretic formulation of this is closely related to the
seminal work of \citet{Robins1986, Robins1987, Robins1989,
  Robins1997}, but is, we consider, more readily interpretable.

The plan of the paper is as follows.  In \secref{description} we
detail our notation, and describe the {\em $G$-recursion algorithm\/}
for evaluating an interventional strategy.  We next discuss the
problem of identifiability, which asks when observational data can be
used to evaluate a strategy.  Distinguishing between the observational
and interventional regimes, we highlight the need for conditions that
would allow us to transfer information across regimes, and thus
support observational evaluation of an interventional strategy.

In \secref{dtapproach} we describe the decision-theoretic framework,
by means of which we can formulate such conditions formally in a
simple and comprehensible way, and so address our questions.  In
particular, we show how the language and calculus of conditional
independence supply helpful tools that we can exploit to attack the
problem of evaluating a strategy from observational data.

In \secref{ss} we introduce \emph{simple stability}, the most
straightforward condition allowing us to evaluate a strategy, by means
of $G$-recursion, from observational data.  However, in many problems
this condition is not easily defensible, so in \secref{sig} we explore
other conditions, in particular conditions we term \emph{sequential
  randomization} and \emph{sequential irrelevance}.  We investigate
when these are sufficient to induce simple stability (and therefore
observational evaluation of a strategy), and discuss their
limitations.  In particular, we show that, when all variables are
discrete, we can drop the requirement of {\em positivity\/} that is
otherwise required to deduce simple stability when sequential
irrelevance holds.  \Countref{cts}, as well as \countref{appb} and
\countref{discretesi} in the Appendix, shows the need for positivity
in more general problems.  Section~\ref{sec:summary} presents some
concluding comments.

\section{A sequential decision problem}
\label{sec:description}

We are concerned with evaluating a specified multistage procedure that
aims to affect a specific outcome variable of interest through a
sequence of interventions, each responsive to observations made thus
far.  As an example we can take the case of HIV disease.  We consider
evaluating strategies that, aiming to suppress the virus and stop
disease progression, recommend when to initiate antiretroviral therapy
for HIV patients based on their history record.  This history will
take into account the CD4 count \citep{Sterne}, as well as additional
variables relevant to the disease.

\subsection{Notation and terminology}
\label{sec:notation}

We consider two sets of variables: $\cal{L}$, a set of
\emph{observable} variables, and $\cal{A}$, a set of \emph{action}
variables.  We term the variables in $\cal{L} \cup \cal{A}$
\emph{domain variables}.  An alternating ordered sequence ${\cal I}:=
(L_1, A_1,\ldots, L_n, A_n, L_{n+1}\equiv Y)$ with $L_i\subseteq{\cal
  L}$ and $A_i\in{\cal A}$ defines an \emph{information base}, the
interpretation being that the specified variables are observed in this
time order.  We will adopt notation conventions such as $(L_1, L_2)$
for $L_1 \cup L_2$, $\overline{L}_i$ for $(L_1, \ldots,L_i)$, \etc.

The observable variables ${\cal L}$ represent initial or intermediate
symptoms, reactions, personal information, \etc, observable between
consecutive treatments, over which we have no direct control; they are
perceived as generated and revealed by Nature.  The action variables
${\cal A}$ represent the treatments, which we could either control by
external intervention, or else leave to Nature to determine.  Thus at
each stage $i$ we will have a realization of the random variable (or
set of random variables) $L_i \subseteq \cal{L}$, followed by a value
for the variable $A_i \in \cal{A}$.  After the realization of the
final $A_n \in \cal{A}$, we will observe the outcome variable $L_{n+1}
\in \cal{L}$, which we also denote by $Y$.

For any stage $i$, a configuration $h_i :=(l_1, a_1, \ldots, a_{i-1},
l_i)$ of the variables $(L_1, A_1,\ldots, A_{i-1}, L_i)$ constitutes a
\emph{partial history}.  A clearly described way of specifying, for
each action $A_i$, its value $a_i$ as a function of the partial
history $h_i$ to date defines a {\em strategy\/}: the values
$(\overline l_i, \overline a_{i-1})$ of the earlier domain variables
$(\overline L_i, \overline A_{i-1}) $ can thus be taken into account
in determining the current and subsequent actions.

In a \emph{static}, or \emph{atomic}, strategy, the sequence of
actions is predetermined, entirely unaffected by the information
provided by the $L_i$'s.  In a {\em non-randomized dynamic strategy\/}
we specify, for each stage $i$ and each partial history $h_i$, a fixed
value $a_i$ of $A_i$, that is then to be applied.  We can also
consider {\em randomized strategies\/}, where for each stage $i$ and
associated partial history $h_i$ we specify a probability distribution
for $A_i$, so allowing randomization of the decision for the next
action.  In this paper we consider general randomized strategies,
since we can regard static and non-randomized strategies as special
cases of these.  Then all the $L_i$'s and $A_i$'s have the formal
status of random variables.  We write \eg\ $\E(L_i \mid
\overline{A}_{i-1},\overline{L}_{i-1}\,;\, s)$ to denote any version
of the conditional expectation $\E(L_i \mid
\overline{A}_{i-1},\overline{L}_{i-1})$ under the joint distribution
$\pr_{s}$ generated by following strategy $s$, and ``\as{\pr_s}'' to
denote that an event has probability 1 under $\pr_{s}$.

\subsection{Evaluating a strategy}
\label{sec:eas}

Suppose we want to identify the effect of some strategy $s$ on the
outcome variable $Y$: we then need to be able to assess the overall
effect that the action variables have on the distribution of $Y$.  An
important application is where we have a loss $L(y)$ associated with
each outcome $y$ of $Y$, and want to compute the expected loss
$\E\{L(Y)\}$ under the distribution for $Y$ induced by following
strategy $s$.  We shall see in \secref{ss} below that, if we know or
can estimate the conditional distribution, under this strategy, of
each observable variable $L_i$ ($i=1,\ldots,n+1$) given the preceding
variables in the information base, then we would be able to compute
$\E\{L(Y)\}$.  Following this procedure for each contemplated
strategy, we could compare the various strategies, and so choose that
minimizing expected loss.

In order to evaluate a particular strategy of interest, we need to be
able to mimic the experimental settings that would give us the data we
need to estimate the probabilistic structure of the domain variables.
Thus suppose that we wish to evaluate a specified non-randomized
strategy for a certain patient $P$, and consider obtaining data under
two different scenarios.

The first scenario corresponds to precisely the strategy that we wish
to evaluate: that is, the doctor knows the prespecified plan defined
by the strategy, and at each stage $i$, taking into account the
partial history $h_i$, he allocates to patient $P$ the treatment that
the strategy recommends.  The expected loss $\E\{L(Y)\}$ computed
under the distribution of $Y$ generated by following this strategy is
exactly what we need to evaluate it.

Now consider a second scenario.  Patient $P$ does not take part in the
experiment described above, but it so happens he has received exactly
the same sequence of treatments that would be prescribed by that
strategy.  However, the doctor did not decide on the treatments using
the strategy, but based on a combination of criteria, that might have
involved variables beyond the domain variables $\cal{L} \cup \cal{A}$.
For example, the doctor might have taken into account, at each stage,
possible allergies or personal preferences for certain treatments of
patient $P$, variables that the strategy did not encompass.

Because these extra variables are not recorded in the data, the
analyst does not know them.  Superficially, both scenarios appear to
be the same, since the variables recorded in each scenario are the
same.  However, without further assumptions there is no reason to
believe that they have arisen from the same distribution.

We call the regime described in the first scenario above an
\emph{interventional regime}, to reflect the fact that the doctor was
intervening in a specified fashion (which we assume known to the
analyst), according to a given strategy for allocating treatment.  We
call the regime described in the second scenario an
\emph{observational regime}, reflecting the fact that the analyst has
just been observing the sequence of domain variables, but does not
know just how the doctor has been allocating treatments.

Data actually generated under the interventional regime would provide
exactly the information required to evaluate the strategy.  However,
typically the data available will not have been generated this
way---and in any case there are so many possible strategies to
consider that it would not be humanly possible to obtain such
experimental data for all of them.  Instead, the analyst may have
observed how patients (and doctors) respond, in a single, purely
observational, regime.  Direct use of such observational data, as if
generated by intervention, though tempting, can be very misleading.
For example, suppose the analyst wants to estimate, at each stage $i$,
the conditional distribution of $L_i$ given $(\overline L_{i-1},
\overline A_{i-1})$ in the interventional regime (which he has not
observed), using data from the observational regime (which he has).
Since all the variables in this conditional distribution have been
recorded in the observational regime, he might instead estimate (as he
can) the conditional distribution of $L_i$ given $(\overline L_{i-1},
\overline A_{i-1})$ in the observational regime, and consider this as
a proxy for its interventional counterpart.  However, since the doctor
may have been taking account of other variables, that the analyst has
not recorded and so can not adjust for, this estimate will typically
be biased, often seriously so.  One of the main aims of this paper is
to consider conditions under which the bias due to such potential
confounding disappears.

For simplicity, we assume that all the domain variables under
consideration can be observed for every patient.  However, the context
in which we observe these variables will determine if and how we can
use the information we collect.  The decision-theoretic approach we
describe below takes into account the different circumstances of the
different regimes by introducing a parameter to identify which regime
is under consideration at any point.  In order to tackle issues such
as the potential for bias introduced by making computations under a
regime distinct from that we are interested in evaluating, we will
need to make assumptions relating the probabilistic behaviours under
the differing regimes.  Armed with such understanding of the way the
regimes interconnect, we can then investigate whether, and if so how,
we can transfer information from one regime to another.

\subsection{Consequence of a strategy}

We seek to calculate the expectation $\E\{k(Y)\,;\, {s}\}$ (always
assumed to exist) of some given function $k(\cdot)$ of $Y$ in a
particular interventional regime $s$; for example, $k(\cdot)$ could be
a loss function, $k(y) \equiv L(y)$, associated with the outcome of
$Y$.  We shall use the term {\em consequence\/} of $s$ to denote the
expectation $\E\{k(Y)\,;\, {s}\}$ of $k(Y)$ under the contemplated
interventional regime $s$.

Assuming $(L_1,A_1, \ldots, L_N, A_N, Y)$ has a joint density in
interventional regime $s$, we can factorize it as:
\begin{equation}
  \label{eq:factor0}
  p(y, \overline{l}, \overline a\,;\, {s}) = \left\{\prod_{i=1}^{n+1}
    p(l_i \mid \overline{l}_{i-1},  \overline{a}_{i-1}\,;\, {s})\right\} \times
  \left\{\prod_{i=1}^n p(a_i \mid \overline{l}_{i},
    \overline{a}_{i-1}\,;\, {s})\right\}
\end{equation}
with $l_{n+1} \equiv y$.

\subsubsection{\texorpdfstring{$G$}{TEXT}-recursion}
\label{sec:gr}

If we knew all the terms on the right-hand side of \eqref{factor0}, we
could in principle compute the joint density for $(Y, \overline L,
\overline A)$ under strategy $s$, hence, by marginalization, the
density of $Y$, and finally the desired consequence $\E\{k(Y);s\}$.
However, a more efficient way to compute this is by means of the {\em
  $G$-computation\/} formula introduced by \citet{Robins1986}.  Here
we describe the recursive formulation of this formula, {\em
  $G$-recursion\/}, generalising the argument in the discrete case
presented by \citet{dawiddidelez}.

Let $h$ denote a partial history of the form $(\overline l_i,
\overline a_{i-1})$ or $(\overline l_i, \overline a_i)$ ($0 \leq i
\leq n+1)$.  We denote the set of all partial histories by ${\cal H}$.
Fixing a regime $s \in \cS$, define a function $f$ on ${\cal H}$ by:
\begin{equation}
  \label{eq:fi0}
  f(h) := \E\{k(Y) \mid h\,;\, s\}.
\end{equation}

\paragraph{Note:} When we are dealing with non-discrete distributions
(and also in the discrete case when there are non-trivial events of
$\pr_s$-probability $0$), the conditional expectation on the
right-hand side of \eqref{fi0} will not be uniquely defined, but can
be altered on a set of histories which has $\pr_s$-probability $0$.
Thus we are in fact requiring, for each $i$:
\begin{equation}
  \label{eq:fi1}
  f(\overline L_i,  \overline A_{i}) := \E\{k(Y) \mid \,\overline L_i,  \overline A_{i};\, s\}\quad \as{\pr_s}
\end{equation}
(and similarly when the argument is $(\overline L_i, \overline
A_{i-1})$).  And we allow the left-hand side of \eqref{fi0} to be {\em
  any\/} selected version of the conditional expectation on the
right-hand side.

For any versions of these conditional expectations, applying the law
of repeated expectation yields:
\begin{eqnarray}
  \label{eq:mainrecura}
  f(\overline L_i, \overline A_{i-1}) &=&  \E\left\{f(\overline L_i, \overline A_i)\, \mid \, \overline L_{i}, \overline A_{i-1}\,;\, s)\right\}\quad \as{\pr_s}\\
  \label{eq:mainrecurl}
  f(\overline L_{i-1}, \overline A_{i-1})& =& \E\left\{f(\overline L_i, \overline A_{i-1}\,  \mid\,  \overline L_{i-1}, \overline A_{i-1}\,;\,s)\right\} \as{\pr_s}.   
\end{eqnarray}
For $h$ a full history $(\overline l_n, \overline a_n, y)$, we have
$f(h) = k(y)$.  Using these starting values, by successively
implementing \eqref{mainrecura} and \eqref{mainrecurl} in turn,
starting with \eqref{mainrecurl} for $i = n+1$ and ending with
\eqref{mainrecurl} for $i=1$, we step down through ever shorter
histories until we have computed $f(\emptyset) = \E\{k(Y)\,;\,s\}$,
the consequence of regime~$s$.  Note that this equality is only
guaranteed to hold almost surely, but since both sides are constants
they must be the same constant.  In particular, it can not matter
which version of the conditional expectations we have chosen in
conducting the above recursion: in all cases we will exit with the
desired consequence $ \E\{k(Y)\,;\,s\}$.

\subsection{Using observational data}
\label{sec:obsdata}

In order to compute $\E\{k(Y)\,;\,s\}$, whether directly from
\eqref{factor0} or using $G$-recursion, \eqref{mainrecura} and
\eqref{mainrecurl}, we need (versions of) the following conditional
distributions under $\pr_s$:

\begin{enumerate}[label=(\roman{enumi})]
\item\label{it:ai} $A_i \,\mid\, \overline L_{i}, \overline A_{i-1}$,
  for $i = 1, \ldots,n$.
\item\label{it:li} $L_i \,\mid\, \overline L_{i-1}, \overline
  A_{i-1}$, for $i = 1, \ldots, n+1$.
\end{enumerate}

Since $s$ is an interventional regime, corresponding to a well-defined
(possibly randomized) treatment strategy, the conditional
distributions in \itref{ai} are fully specified by the treatment
protocol.  So we only need to get a handle on each term of the form
\itref{li}.  However, since we have not implemented the strategy $s$,
we do not have data directly relevant to this.  Instead, we might be
tempted to use its observational counterpart, \ie\ a version of the
conditional distribution of $l_i\, \mid \overline L_{i-1}, \overline
A_{i-1}$ in the observational regime $\pr_o$, which is (in principle)
estimable from observational data.  This will generally be a dangerous
ploy, since we are dealing with two quite distinct regimes, with
strong possibilities for confounding and other biases in the
observational regime; however, it can be justifiable if we can impose
suitable extra conditions, relating the probabilistic behaviours of
the different regimes.  We therefore now turn to a description of a
general ``decision-theoretic'' framework that is useful for expressing
and manipulating such conditions.

\section{The decision-theoretic approach}
\label{sec:dtapproach}
\label{sec:dt}
In the decision-theoretic approach to causal inference, we proceed by
making suitable assumptions relating the probabilistic behaviours of
stochastic variables across a variety of different regimes.  These
could relate to different locations, time-periods, or, in this paper,
contexts (observational/interventional regimes) in which observations
can be made.  We denote the set of all regimes under consideration by
$\cS$.  We introduce a non-stochastic variable $\sigma$, the
\emph{regime indicator}, taking values in $\cS$, to index these
regimes and their associated probability distributions.  Thus $\sigma$
has the logical status of a parameter, rather than a random variable:
it specifies which (known or unknown) joint distribution is operating
over the domain variables ${\cal L}\cup{\cal A}$.  Any probabilistic
statement about the domain variables must, explicitly or implicitly,
be conditional on some specified value $s\in{\cal S}$ for $\sigma$.

We focus here on the case that we want to make inference about one or
more interventional regimes on the basis of data generated under an
observational regime.  So we take $\cS=\{o\} \cup \cS^*$, where $o$ is
the observational regime under which data have been gathered, and
$\cS^*$ is the collection of contemplated interventional strategies
with respect to a given information base $(L_1, A_1,\ldots, L_N, A_N,
Y)$.

\subsection{Conditional independence}
\label{sec:ci}

In order to address the problem of making inference from observational
data we need to assume (and justify) some relationships between the
probabilistic behaviours of the variables in the differing regimes,
interventional and observational.  These assumptions will typically
relate certain conditional distributions across different regimes.
The notation and calculus of {\em conditional independence\/} (CI)
turn out to be well-suited to express and manipulate such assumptions.

\subsubsection{Conditional independence for stochastic variables}
\label{sec:stocci}
Let $X, Y, Z, \ldots$ be random variables defined on the same
probability space $(\Omega,{\cal A},P)$.  We write $\ind X Y Z\,\,\,
[P]$, or just $\ind X Y Z$ when $P$ is understood, to denote that {\em
  $X$ is independent of $Y$ given $Z$\/} under $P$: this can be
interpreted as requiring that the conditional distribution, under $P$,
of $X$, given $Y=y$ and $Z=z$, depends only on $y$ and not further on
the value $z$ of $Z$.  More formally, we require that, for any bounded
real measurable function $h(X)$, there exists a measurable function
$w(Z)$ such that
\begin{equation}
  \label{eq:stocci}
  \E\{h(X) \cd Y, Z\} = w(Z)\quad\as{P}.
\end{equation} 

Stochastic CI so defined has various general properties, of which the
most important are the following---which can indeed be used as axioms
of an independent ``calculus of CI''
\citep{Dawid1979,Dawid2001i,pearl1988}.

\begin{theorem}
  \label{thm:axioms}
  \quad
  \begin{enumerate}[label=(\roman{enumi})]
  \item[P1](Symmetry) \ind {X} {Y} {Z} $\Rightarrow$ \ind {Y} {X} {Z}
  \item[P2]\ind {X} {Y} {X}
  \item[P3](Decomposition) \ind {X} {Y} {Z}, $W \preceq Y$
    $\Rightarrow$ \ind {X} {W} {Z}
  \item[P4](Weak Union) \ind {X} {Y} {Z}, $W \preceq Y$ $\Rightarrow$
    \ind {X} {Y} {(W,Z)}
  \item[P5](Contraction) \ind {X} {Y} {Z} and \ind {X} {W} {(Y,Z)}
    $\Rightarrow$ \ind {X} {(Y,W)} {Z}
  \end{enumerate}
\end{theorem}
(Here $W \preceq Y$ is used to denote that $W =f(Y)$ for some
measurable function $f$.)

These properties can be shown to hold universally for random variables
on a common probability space \citep{nayia:thesis}.

\subsubsection{Extended conditional independence}
\label{sec:eci}

We can generalize the property $\ind X Y Z$ by allowing either or both
of $Y, Z$ to be or contain non-stochastic elements, such as parameters
or regime indicators \citep{Dawid1979, Dawid1980, Dawid1998}: in this
case we talk of {\em extended conditional independence\/}.  Thus let
$\sigma$ denote the non-stochastic regime indicator.  Informally, we
interpret $\ind X \sigma Z$ as saying that the conditional
distribution of $X$, given $Z=z$, under regime $\sigma=s$, depends
only on $z$ and not further on the value $s$ of $\sigma$; that is to
say, the conditional distribution of $X$ given $Z$ is the same in all
regimes.  Note that this is exactly the form of ``causal assumption'',
allowing transfer of probabilistic information across regimes, that we
might wish to apply.

More formally, let $\{P_s: s \in{\cal S}\}$ be a family of
distributions, and $X$, $Y$, $Z$,\ldots random variables, on a measure
space $(\Omega, {\cal A})$.  We introduce the non-stochastic regime
indicator variable $\sigma$ taking values in ${\cal S}$, and interpret
conditioning on $\sigma=s$ to mean that we are computing under
distribution $P_s$.

\begin{definition}
  \label{def:nstoc3}
  We say that {\itshape $X$ is (conditionally) independent of $Y$
    given $(Z,\sigma)$} and write \ind {X}{Y}{(Z,\sigma)}, if for any
  bounded real measurable function $h(X)$, there exists a function
  $w(\sigma,Z)$, measurable in $Z$, such that, for all $s\in{\cal S}$,
  \begin{equation*}
    \E\{h(X)\mid Y,Z\,;\,s\}=w(s,Z) \quad \as{P_{s}}.
  \end{equation*} 
\end{definition}

\begin{definition}
  \label{def:nstoc1}
  We say that \emph{$X$ is (conditionally) independent of $(Y,\sigma)$
    given $Z$}, and write $\ind {X} {(Y,\sigma)} {Z}$, if for any
  bounded real measurable function $h(X)$, there exists a measurable
  function $w(Z)$ such that, for all $s \in {\cal S}$,
  \begin{equation}
    \label{eq:nonstocci}
    \E\{h(X)\mid Y, Z\,;\,s\}=w(Z) \quad \as{P_s}.
  \end{equation}
\end{definition}

\begin{remark}
  \quad\\[-3ex]
  \begin{enumerate}
  \item Note the similarity of \eqref{nonstocci} to \eqref{stocci}.
    In particular the function $w(Z)$ must not depend on the regime
    $s\in{\cal S}$ operating.
  \item When $X$, $Y$ and $Z$ are discrete random variables, $\ind X
    {(Y,\sigma)} Z$ if and only if there exists a function $w(X,Z)$
    such that, for any $s \in {\cal S}$,
    \begin{equation*}
      \pr(X=x\cd Y=y, Z=z\,;\,s)=w(x,z)
    \end{equation*} 
    whenever $\pr(Y=y, Z=z\,;\,s)>0$.
  \item For each $s\in{\cal S}$, the equality in \eqref{nonstocci} is
    permitted to fail on a set $A_s$, which may vary with $s$, that
    has probability $0$ under $P_s$.
  \item The requirement of \eqref{nonstocci} is that there exist a
    single function $w(Z)$ that can serve as the conditional
    expectation of $h(X)$ given $(Y,Z)$ in every distribution $P_s$;
    but this does not imply that any version of this conditional
    expectation under one value of $s$ will serve for all values of
    $s$: see \countref{appb} for a counter-example, and
    \citet{apd:misl} for cases where a lack of understanding of
    similar problems associated with null events has led to serious
    errors.  However we can sometimes escape this problem by imposing
    an additional {positivity} condition: see \secref{positivity}
    below.
  \end{enumerate}
\end{remark}

\subsubsection{Connexions}
\label{sec:connex}

In this section we impose the additional condition that the set ${\cal
  S}$ of possible regimes be finite or countable, and endow it with
the $\sigma$-field ${\cal F}$ of all its subsets.

We can construct the product measure space $(\Omega^*,{\cal A}^*) :=
(\Omega \times {\cal S}, {\cal A}\otimes{\cal F})$, and regard all the
stochastic variables $X, Y, Z, \ldots$ as defined on $(\Omega^*,{\cal
  A}^*)$; moreover $\sigma$ can also be considered as a random
variable on $(\Omega^*,{\cal A}^*)$.

Let $\Pi$ be a probability measure on ${\cal S}$, arbitrary subject
only to giving positive probability $\pi(s)>0$ to each point
$s\in{\cal S}$; and define, for any $A^* \in {\cal A}^*$:
\begin{equation}
  \label{eq:p*}
  P^*(A) = \sum_{s\in{\cal S}} \pi(s) P_s(A_s)
\end{equation}
where $A_s =\{\omega\in\Omega: (\omega,s)\in A\}$.  Under $P^*$ the
marginal distribution of $\sigma$ is $\Pi$, while the conditional
distribution over $\Omega$, given $\sigma = s$, is $P_s$.  It is then
not hard to show \citep{nayia:thesis} that $\ind X Y {(Z,\sigma)}$ in
the extended sense of \defref{nstoc3} if and only if the purely
stochastic interpretation of the same expression holds under $P^*$;
and similarly for \defref{nstoc1}.  It follows that, for the
interpretations of extended conditional independence given in
\secref{eci}, we can continue to apply all the properties P1--P5 of
\thmref{axioms}.  Any argument so constructed, in which all the
premisses and conclusions are so interpretable, will be valid---even
when some of the intermediate steps are not so interpretable (\eg,
they could have the form $\ind \sigma X Y$).

For the purposes of this paper we will only ever need to compare two
regimes at a time: the observational regime $o$ and one particular
interventional regime $s$ of interest.  Then the properties P1--P5 of
conditional independence can always be applied, and equip us with a
powerful machinery to pursue identification of interventional
quantities from observational data.
 
\subsubsection{Graphical representations}
\label{sec:graphrep}

Graphical models in the form of {\em influence diagrams\/} (IDs) can
sometimes be used to represent collections of conditional independence
properties amongst the variables (both stochastic and non-stochastic)
in a problem \citep{lauritzen, Dawid2002, cowell}.  We can then use
graphical techniques (in particular, the {\em $d$-separation\/}, or
the equivalent {\em moralization\/}, criterion) to derive, in a visual
and transparent way, implied conditional independence properties that
follow from our assumptions.  However, a graphical representation is
not always possible and never essential: all that can be achieved
through the graph-theoretic properties of IDs, and more, can be
achieved using the calculus of conditional independence (properties
P1--P5).

\section{Simple stability}
\label{sec:ss}

We now use CI to express and explore some conditions that will allow
us to perform $G$-recursion for the strategy of interest on the basis
of observational data.

Consider first the conditional distribution \itref{ai} of $A_i \mid
\overline L_{i}, \overline A_{i-1}\,;\, s$ as needed for
\eqref{mainrecura}.  This term requires knowledge of the mechanism
that allocates the treatment at stage $i$ in the light of the
preceding variables in the information base.  We assume that, for an
interventional regime $s \in \cS^*$, this distribution (degenerate for
a non-randomized strategy) will be known {\em a priori\/} to the
analyst, as it will be encoded in the strategy.  In such a case we
call $s \in \cS^*$ a {\em control strategy\/} (with respect to the
information base ${\cal I} = (L_1, A_1,\ldots, L_N, A_N, Y)$).

Next we consider how we might gain knowledge of the conditional
distribution \itref{li} of $L_i \mid \overline L_{i-1}, \overline
A_{i-1}\,;\, s$, as required for \eqref{mainrecurl}.  This
distribution is unknown, and we need to explore conditions that will
enable us to identify it from observational data.  As different
distributions for the random variables in the information base apply
in the different regimes, the distribution of $L_i$ given $(\overline
L_{i-1}, \overline A_{i-1})$ will typically depend on the regime
operating.

\begin{definition}
  \label{def:simplestability}
  We say that the problem exhibits {\em simple
    stability\/}\footnote{This definition is slightly weaker than that
    of \citet{dawiddidelez}, as we are only requiring a common version
    of the corresponding conditional expectations between each single
    control strategy and the observational regime.  We do not require
    that there exist one function that can serve as common version
    across all regimes simultaneously.} with respect to the
  information base ${\cal I} = (L_1,A_1, \ldots, L_n,A_n, Y)$ if, for
  each $s\in{\cal S}^*$, with $\sigma$ denoting the non-random regime
  indicator taking values in $\{o,s\}$:
  \begin{equation}
    \label{eq:st}
    \ind {L_i} {\sigma} {(\overline L_{i-1}, \overline A_{i-1})} \quad (i =1, \ldots, n+1).
  \end{equation}
\end{definition}

Formally, simple stability requires that, for any bounded measurable
function $f(L_i)$, there exist a single random variable $W =
w(\overline L_{i-1}, \overline A_{i-1})$ that serves as a version of
each of the conditional expectations $\E\{f(L_i)\mid (\overline
L_{i-1}, \overline A_{i-1})\,;\,o\}$ and $\E\{f(L_i)\mid (\overline
L_{i-1}, \overline A_{i-1})\,;\,s\}$.  This property then extends to
conditional expectations of functions of the form $f(\overline L_i,
\overline A_{i-1})$.  In particular, this apparently\footnote{but see
  \secref{positivity} below} supports identification of the right-hand
side of \eqref{mainrecurl} with its observational counterpart, so
allowing observational estimation estimation of this expression.

Simple stability is a very strong assumption, and will be tenable only
in very special cases.  It will be satisfied if, in the observational
regime, the action variables are physically sequentially randomized:
then all unobserved potential confounding factors will, on average, be
balanced between the treatment groups.  Alternatively, we might accept
simple stability if, in the observational regime, the allocation of
treatment is decided taking into account only the domain variables in
the information base and nothing more: for example, if we are
observing a doctor whose treatment decisions are based only on the
domain variables we are recording, and no additional unrecorded
information.

The ID describing simple stability \eqref{st} for $i = 1, 2, 3$ is
shown in \figref{simpdiag}.  The specific property \eqref{st} is
represented by the absence of arrows from $\sigma$ to $L_1$, $L_2$,
and $L_3 \equiv Y$.

\begin{figure}[h]
  \begin{center}
    \resizebox{2.8in}{!}{\includegraphics{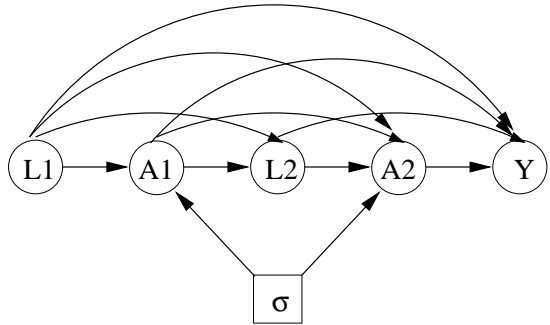}}
    \caption{Stability}
    \label{fig:simpdiag}
  \end{center}
\end{figure}

\subsection{Positivity}
\label{sec:positivity}

We have indicated that simple stability might allow us to identify the
consequence of a control strategy $s$ on the basis of data from the
observational regime $o$.  However, while this condition ensures the
existence of a common version of the relevant conditional expectation
valid for both regimes, deriving this function from the observational
regime alone might be problematic, because versions of the same
conditional expectation can differ on events of probability $0$, and
we have not ruled out that an event having probability $0$ in one
regime might have positive probability in another.  Thus we can only
obtain the desired function from the observational regime on a set
that has probability 1 in the observational regime; and this might not
have probability 1 in the intervententional regime---see
\countref{appb} in the Appendix for a simple example of this.

To evade this problem, we can impose a condition requiring an event to
have zero probability in the interventional regime whenever it has
zero probability in the observational regime:

\begin{definition}
  \label{def:positivity}
  We say the problem exhibits {\em positivity\/} or \emph{absolute
    continuity} if, for any interventional regime $s \in \cS^*$, the
  joint distribution of $(\overline{L_n}, \overline{A_n}, Y)$ under
  $\pr_s$ is absolutely continuous with respect to that under $\pr_o$,
  \ie:
  \begin{equation}
    \label{eq:pos}
    \pr_{s}(E) > 0 \Rightarrow \pr_{o}(E)> 0
  \end{equation}
  for any event $E$ defined in terms of $(\overline{L_n},
  \overline{A_n}, Y)$.
\end{definition}

Suppose we have both simple stability and positivity, and consider a
bounded function $h(L_i)$.  Let $W = w(\overline L_{i-1}, \overline
A_{i-1})$ be any variable that serves both as a version of
$\E\{h(L_i)\mid \overline L_{i-1}, \overline A_{i-1}\,;\,o\}$ and as a
version of $\E\{h(L_i)\mid \overline L_{i-1}, \overline
A_{i-1}\,;\,s\}$; such a variable is guaranteed to exist by
\eqref{st}.  Let $V = v(\overline L_{i-1}, \overline A_{i-1})$ be any
version of $\E\{h(L_i)\mid \overline L_{i-1}, \overline
A_{i-1}\,;\,o\}$.  Since $W$ too is a version of $\E\{h(L_i)\mid
\overline L_{i-1}, \overline A_{i-1}\,;\,o\}$, $V=W, \as{\pr_o}$.
Hence, by \eqref{pos}, $V=W, \as{\pr_s}$.  But since $W$ is a version
of $\E\{h(L_i)\mid \overline L_{i-1}, \overline A_{i-1}\,;\,s\}$, so
too must be $V$.  So we have shown that any version of a conditional
expectation calculated under $\pr_o$ will also serve this purpose
under $\pr_s$.  In particular, when effecting the $G$-computation
algorithm of \secref{gr}, in \eqref{mainrecurl} we are fully justified
in replacing the conditional expectation under $\pr_s$ by (any version
of) its counterpart under $\pr_o$---which we can in principle estimate
from observational data.

\subsubsection{Difficulties with continuous actions}
\label{sec:probcont}

When all variables are discrete, positivity will hold if and only if
every partial history that can occur with positive probability in the
interventional regime also has a positive probability in the
observational regime.  In particular, this will hold for every
interventional regime if every possible partial history can occur with
positive probability in the observational regime.

Even in this case we might well need vast quantities of observational
data to get good estimates of all the probabilities needed for
substitution into the $G$-recursion algorithm---that is the reason for
our qualification ``in principle'' at the end of \secref{positivity}.
In practice, even under positivity we would generally need to impose
some smoothness or modelling assumptions to get reasonable estimates
of the required observational distributions.  However we do not
explore these issues here, merely noting that, given enough data to
estimate these observational distributions, positivity allows us to
transfer them to the interventional regime.

When however we are dealing with continuous action variables---as, for
example, the dose of a medication---the positivity condition may
become totally unreasonable.  For a very simple example, consider a
single continuous action variable $A$ and response variable $Y$.  We
might want to transfer the conditional expectation $\E(Y\,|\,A)$ from
the observational regime $o$, in which $A$ arises from a continuous
distribution, to an interventional regime $s$, in which it is set to a
fixed value, $A=a_0$.  However, if we take any version of
$\E(Y\,|\,A;o)$ and change it, to anything we want, at the single
point $A=a_0$, we will still have a version of $\E(Y\,|\,A;o)$.  So we
are unable to identify the desired $\E(Y\,|\,A;s)$ This is due to the
failure of positivity, since the $1$-point interventional distribution
of $A$ is not absolutely continuous with respect to the continuous
observational distribution of $A$.  Positivity here would require that
there be a positive probability of observing the exact value $a_0$ in
the observational regime.  But it would not generally be reasonable to
impose such a condition, and quite impossible to do so for every value
$a_0$, that we might be potentially interested in setting for $A$.

In such a case we might make progress by imposing further structure,
such as a model for $\E(Y\,|\,A;o)$ that is a continuous function of
$A$, so identifying a preferred version of this.  Here however we
shall avoid such problems by only considering problems in which all
action variables are discrete.  Then we shall have positivity whenever
every action sequence $\overline a$ having positive interventional
probability also has positive observational probability, and the
(uniquely defined) conditional interventional distribution of all the
non-action variables, given $\overline A = \overline a$, is absolutely
continuous with respect to its observational counterpart.  This will
typically not be an unreasonable requirement.  We note that our set-up
is still more general than the usual formulations of $G$-recursion,
which explicitly or implicitly assume that all variables are discrete.

\section{Sequential ignorability}
\label{sec:sig}

As we have alluded, simple stability will often not be a compelling
assumption, for example because of the suspected presence of
unmeasured confounding variables, and we might not be willing to
accept it without further justification.  Here we consider conditions
that might seem more acceptable, and investigate when these will,
after all, imply simple stability---thus supporting the application of
$G$-recursion.
 
\subsection{Extended stability and extended positivity}
\label{sec:es}
 
Let $\mathcal{U}$ denote a set of variables that, while they might
potentially influence actions taken under the observational regime,
are not available to the decision maker, and so are not included in
his information base ${\cal I}:= (L_1, A_1,\ldots, L_n, A_n,
L_{n+1}\equiv Y)$.  We define the \emph{extended information base}
${\cal I}':= (L_1, U_1, A_1,\ldots, L_n, U_n, A_n, L_{n+1})$, with
$U_i$ denoting the variables in ${\cal U}$ realized just before action
$A_i$ is taken.  However, while thus allowing $U_i$ to influence $A_i$
in the observational regime, we still only consider interventional
strategies where there is no such influence---since the decision maker
does not have access to the $(U_i)$.  This motivates an extended
formal definition of ``control strategy'' in this context:

\begin{definition}[Control strategy]
  A regime $s$ is a {\em control strategy\/} if
  \label{cond:cont}
  \begin{equation}
    \label{eq:controlpar}
    \ind {A_i} {\overline{U}_i} {(\overline{L}_i, \overline{A}_{i-1}\,;\, s)} \quad (i = 1,\ldots, n)
  \end{equation}
  and in addition, the conditional distribution of $A_i$, given
  $(\overline L_{i}, \overline A_{i-1})$, under regime $s$, is known
  to the analyst.
\end{definition}
We again denote the set of interventional regimes corresponding to the
control strategies under consideration by ${\cal S}^*$.

\begin{definition}
  \label{def:es}
  We say that the problem exhibits {\em extended stability\/} (with
  respect to the extended information base ${\cal I}'$) if, for any
  $s\in{\cal S}^*$, with $\sigma$ denoting the non-random regime
  indicator taking values in $\{o,s\}$:
  \begin{equation}
    \label{eq:es}
    \ind {(L_i,U_i)} \sigma ({\overline{L}_{i-1}, \overline{U}_{i-1}, \overline{A}_{i-1}}) \quad (i =1, \ldots, n+1).    
  \end{equation}
\end{definition}

Extended stability is formally the same as simple stability, but using
a different information base, where $L_i$ is expanded to $(L_i, U_i)$.
The real difference is that the extended information base is not
available to the decision maker in the interventional regime, so that
his decisions can not take account of the $(U_i)$.  An ID faithfully
representing property \eqref{es} for $i = 1, 2, 3$ is shown in
\figref{extdiag}.\footnote{Note that the IDs in this paper differ from
  those in \citet{dawiddidelez}.}  The property \eqref{es} is
represented by the absence of arrows from $\sigma$ to $L_1$, $U_1$,
$L_2$, $U_2$ and $Y$.  However, the diagram does not explicitly
represent the additional property \eqref{controlpar}, which implies
that, {\em when $\sigma = s$\/}, the arrows into $A_1$ from $U_1$ and
into $A_2$ from $U_1$ and $U_2$ can be dropped.

\begin{figure}[h]
  \begin{center}
    \resizebox{2.8in}{!}{\includegraphics{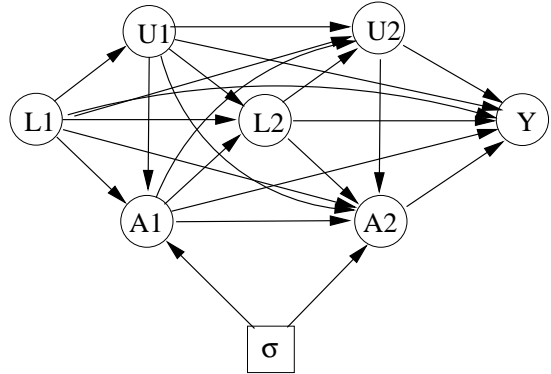}}
    \caption{Extended stability}
    \label{fig:extdiag}
  \end{center}
\end{figure}

To evade problems with events of zero probability, we can extend
\defref{positivity}:
\begin{definition}
  \label{def:extpositivity}
  We say the problem exhibits {\em extended positivity\/} if, for any
  $s \in \cS^*$, the joint distribution of $(\overline{U_n},
  \overline{L_n}, \overline{A_n}, Y)$ under $\pr_s$ is absolutely
  continuous with respect to that under $\pr_o$, \ie
  \begin{equation}
    \label{eq:extpos}
    \pr_{s}(E) > 0 \Rightarrow \pr_{o}(E)> 0
  \end{equation}
  for any event $E$ defined in terms of $(\overline{L_n},
  \overline{U_n}, \overline{A_n}, Y)$.
\end{definition}

\subsection{Sequential randomization}
\label{sec:sr}

Extended stability represents the belief that, for each $i$, the
conditional distribution of $(L_i,U_{i})$, given all the earlier
variables $(\overline L_{i-1}, \overline U_{i-1}, \overline A_{i-1})$
in the extended information base, is the same in the observational
regime as in the interventional regime.  This will typically be
defensible if we can argue that we have included in ${\cal L}\cup{\cal
  U}$ all the variables influencing the actions in the observational
regime.

However extended stability, while generally more defensible than
simple stability, typically does not imply simple stability, which is
what is required to support $G$-recursion.  But it may do so if we
impose additional conditions.  Here and in \secref{si} below we
explore two such conditions.

Our first is the following:

\begin{condition}[Sequential randomization]
  \label{con:pa1}
  \begin{equation}
    \label{eq:pa0}
    \ind {A_i} {\overline U_i} {(\overline L_i, \overline A_{i-1}\,;\, o)} \quad (i = 1, \ldots, n).
  \end{equation}
\end{condition}
Taking account of \eqref{controlpar}, we see that \eqref{pa0} is
equivalent to:
\begin{equation}
  \label{eq:pa1}
  \ind {A_i} {\overline{U}_i} {(\overline{L}_i, \overline{A}_{i-1}\,;\, \sigma)} \quad (i = 1,\ldots, n)
\end{equation}
where $\sigma$ takes values in ${\cal S} = \{o\} \cup {\cal S}^*$.

Under sequential randomization, the observational distribution of
$A_i$, given the earlier variables in the information base, would be
unaffected by further conditioning on the earlier unobservable
variables, ${\overline U_i}$.  Hence the $(U_i)$ are redundant for
explaining the way in which actions are determined in the
observational regime.  While this condition will hold under a control
strategy, in the observational regime it requires that the only
information that has been used to assign the treatment at each stage
is that supplied by the observable variables.  For example, sequential
randomization will hold if the actions are physically sequentially
randomized within all levels of the earlier variables in the
information base.  The following result is therefore unsurprising.
\begin{theorem}
  \label{thm:parner1}
  Suppose we have both extended stability, \eqref{es} and sequential
  randomization, \eqref{pa1}.  Then we have simple stability,
  \eqref{st}.
\end{theorem}

An ID faithfully representing the conditional independence
relationships assumed in \thmref{parner1}, for $i=1,2,3$, is shown in
\figref{parner1}.
\begin{figure}[hbtp]
  \begin{center}
    \resizebox{2.8in}{!}{\includegraphics{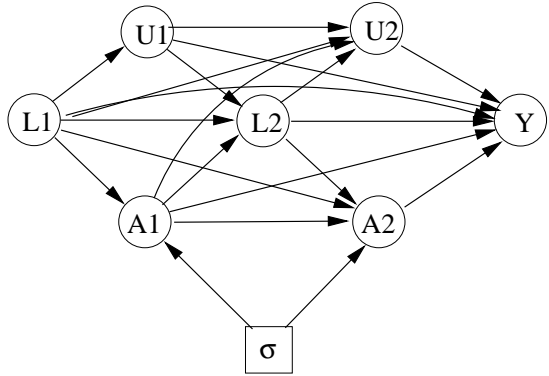}}
    \caption{Sequential randomization}
    \label{fig:parner1}
  \end{center}
\end{figure}
\figref{parner1} can be obtained from \figref{extdiag} on deleting the
arrows into $A_1$ from $U_1$ and into $A_2$ from $U_1$ and $U_2$, so
representing \eqref{pa1}.  (However, as we shall see below in
\secref{si}, in general such ``surgery'' on IDs can be hazardous.)

The conditional independence properties \eqref{st} characterising
simple stability can now be read off from \figref{parner1}, by
applying the $d$-separation or moralization criteria.  For a formal
algebraic proof of \thmref{parner1}, using just the axioms of
conditional independence as given in \theoremref{axioms}, see
Theorem~6.1 of \citet{dawiddidelez}.\footnote{Note that, in either of
  these approaches, we can restrict $\sigma$ to the two values $o$ and
  $s$, so fully justifying treating the non-stochastic variable
  $\sigma$ as if it were stochastic.}

\begin{corollary}
  \label{cor:es}
  Suppose we have extended stability, sequential randomization, and
  extended positivity.  Then we can apply $G$-recursion to compute the
  consequence of a strategy $s\in{\cal S}^*$.
\end{corollary}

\subsection{Sequential irrelevance}
\label{sec:si}

Consider now the following alternative condition:
\begin{condition}[Sequential Irrelevance]
  \label{con:pa1a}
  \begin{equation}
    \label{eq:pa1a}
    \ind {L_i} {\overline U_{i-1}} {(\overline L_{i-1}, \overline A_{i-1}\,;\, \sigma)} \quad (i = 1, \ldots, n+1).
  \end{equation}
\end{condition}

Under sequential irrelevance, in both regimes the conditional
distribution of the observable variable(s) at stage $i$ is unaffected
by the history of unobservable variables up to the previous stage
$i-1$, given the domain variables in the information base up to the
previous stage.  In contrast to \eqref{pa1}, \eqref{pa1a} permits the
unobserved variables that appear in earlier stages to influence the
next action $A_i$ (which can only happen in the observational
regime)---but not the development of the subsequent observable
variables (including the ultimate response variable $Y$).  This
condition will typically hold when at each stage $i$ the unobserved
variable $U_i$, while possibly influencing the next observable
variable $L_i$, does not affect the development of future $L$'s.  An
example might be where the unobservable variables represent the
inclination of the patient to take the treatment: in this case one
might expect $U_i$ to affect the development of $L_i$ but not
subsequent $L$'s.  In general, the validity of this assumption will
have to be justified in the context of the problem under study.

By analogy with the passage from \figref{extdiag} to \figref{parner1},
we might attempt to represent the additional assumption \eqref{pa1a}
by removing from \figref{extdiag} all arrows from $U_j$ to $L_i$
($j<i$).  This would yield \figref{parner1a}.
\begin{figure}[hb]
  \begin{center}
    \resizebox{2.8in}{!}{\includegraphics{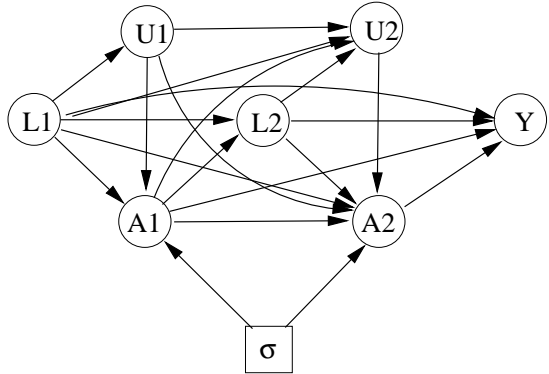}}
    \caption{Sequential irrelevance?}
    \label{fig:parner1a}
  \end{center}
\end{figure}
On applying $d$-separation or moralization to \figref{parner1a} we
could then deduce the simple stability property \eqref{st}.  However,
this approach is not valid, since \figref{parner1a} encodes the
property $\ind {L_2} \sigma {(L_1, A_1)}$, which can not be derived
from \eqref{es} and \eqref{pa1a} using only the ``axioms'' of
\thmref{axioms}.  In fact there is no ID that faithfully represents
the combination of the properties \eqref{es} and \eqref{pa1a}, since
these do not form a recursive system.  And indeed, in full generality,
simple stability is not implied by extended stability, \eqref{es},
together with sequential irrelevance, \eqref{pa1a}, as the following
counter-example demonstrates.

\begin{counterexample}
  \label{count:cts}
  Take $n=1$, ${\cal L} = \emptyset$ and ${\cal U} = \{U\}$.  The
  extended information base is ${\cal I}'=(U, A, Y)$.  We suppose
  that, in both the observational regime $o$ and the interventional
  regime $s$, $Y = 1$ if $A=U$, else $Y=0$.  Also, in each regime, the
  marginal distribution of $U$ is uniform on $[0,1]$.  It remains to
  specify the distribution of $A$, given $U$: we assume that, in
  regime $o$, $A=U$, while in regime $s$, $A$ is uniform on $[0,1]$,
  independently of $U$.

  It is readily seen that $\indo {U} {\sigma}$ and $\ind {Y} {\sigma}
  {(U,A)}$.  Thus we have extended stability, \eqref{es}, as
  represented by the ID of \figref{counter}.
  
  \begin{figure}[hb]
    \begin{center}
      \resizebox{1.2in}{!}{\includegraphics{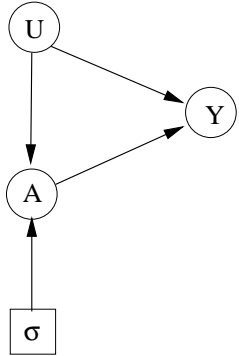}}
      \caption{Counter-example}
      \label{fig:counter}
    \end{center}
  \end{figure}

  Also, since $\indo{U}{A}$ in regime $s$, \eqref{controlpar} holds,
  so $s$ is a control strategy.  Finally, in regime $o$, $Y=1$ \aso,
  while in regime $s$, $Y=0$ \aso.  Because these are both degenerate
  distributions, trivially $\ind {Y} {U} {(A,\sigma)}$, and we have
  sequential irrelevance.  However, because they are different
  distributions, $\ninda {Y} {\sigma} {A}{}$: so we do {\em not\/}
  have simple stability, \eqref{st}.  In particular, we can not remove
  the arrow from $U$ to $Y$ in \figref{counter}, since this would
  encode the false property $\ind {Y} {\sigma} {A}$.  \qed
\end{counterexample}

So, if we wish to deduce simple stability from extended stability and
sequential irrelevance, further conditions, and a different approach,
will be required.

In Theorem~6.2 of \citet{dawiddidelez} it is shown that this result
does follow if we additionally impose the extended positivity
condition of \defref{extpositivity}; and then we need only require
sequential irrelevance, \eqref{pa1a}, to hold for the observational
regime $\sigma = o$.

However, in \secref{dc} below we show that, if we restrict attention
to discrete variables, no further conditions are required for this
result to hold.  In this case, we need only require sequential
irrelevance to hold for the interventional regime $\sigma = s$.

\section{Discrete case}
\label{sec:dc}

In this section we assume all variables are discrete, and denote
$P(\overline A=\overline a, \overline L = \overline l)$ by
$p(\overline a, \overline l)$, \etc

To control null events, we need the following lemma:
\begin{lemma}
  \label{lem:sidc}
  Let all variables be discrete.  Suppose that we have extended
  stability, \eqref{es}, and let $s$ be a control strategy, so that
  \eqref{controlpar} holds.  Then, for any $(\overline{u}_k,
  \overline{l}_k,\overline{a}_k)$ such that
  \begin{description}
  \item[$\bf A_k$:] $p(\overline{l}_k,\overline{a}_k\,;\,s)>0$, and
  \item[$\bf B_k$:]
    $p(\overline{u}_k,\overline{l}_k,\overline{a}_k\,;\,o) > 0 $, we
    have
  \item[$\bf C_k$:]
    $p(\overline{u}_k,\overline{l}_k,\overline{a}_k\,;\,s)>0$.
  \end{description}
\end{lemma}

\begin{proof}
  Let $H_{k}$ denote the assertion that {$\bf A_k$} and {$\bf B_k$}
  imply {$\bf C_k$}.  We establish $H_k$ by induction.

  To start, we note that $H_0$ holds vacuously.

  Now suppose $H_{k-1}$ holds.  Assume further {$\bf A_k$} and {$\bf
    B_k$}.  Together these conditions imply that all terms appearing
  throughout the following argument are positive.

  We have
  \begin{eqnarray}
    p(\overline{u}_{k}, \overline{l}_{k}, \overline{a}_{k}\,;\,s)
    &=& p(\overline{u}_{k}\mid
    \overline{l}_{k},\overline{a}_{k}\,;\,s)\,p(\overline{l}_{k},\overline{a}_{k}\,;\,s)\nonumber
    \\
    &=& p(\overline{u}_{k}\mid
    \overline{l}_{k},\overline{a}_{k-1}\,;\,s)\,
    p(\overline{l}_{k},\overline{a}_{k}\,;\,s)
    \label{eq:sil7}
    \\
    &=& \frac{p(\overline{u}_{k},\overline{l}_{k},\overline{a}_{k-1}\,;\,s)}{p(\overline{l}_{k},\overline{a}_{k-1}\,;\,s)}
    p(\overline{l}_{k},\overline{a}_{k}\,;\,s) \nonumber
    \\
    &=& p(u_{k},l_{k}\mid
    \overline{u}_{k-1},\overline{l}_{k-1},\overline{a}_{k-1}\,;\,s)
    \nonumber
    \\
    &&{}    \times\,\frac{p(\overline{u}_{k-1},\overline{l}_{k-1},\overline{a}_{k-1}\,;\,s)\,p(\overline{l}_{k},\overline{a}_{k}\,;\,s)}{p(\overline{l}_{k},\overline{a}_{k-1}\,;\,s)}\nonumber
    \\
    &=& p(u_{k},l_{k}\mid
    \overline{u}_{k-1},\overline{l}_{k-1},\overline{a}_{k-1}\,;\,o)
    \nonumber
    \\
    &&{}\times\,\frac{p(\overline{u}_{k-1},\overline{l}_{k-1},\overline{a}_{k-1}\,;\,s)\,p(\overline{l}_{k},\overline{a}_{k}\,;\,s)}{p(\overline{l}_{k},\overline{a}_{k-1}\,;\,s)}
    \label{eq:sil9}
    \\
    &=& \frac{p(\overline{u}_{k},\overline{l}_{k},\overline{a}_{k-1}\,;\,o)}{p(\overline{u}_{k-1},\overline{l}_{k-1},\overline{a}_{k-1}\,;\,o)}
    \nonumber
    \\
    &&{}\times    \,\frac{p(\overline{u}_{k-1},\overline{l}_{k-1},\overline{a}_{k-1}\,;\,s)\,p(\overline{l}_{k},\overline{a}_{k}\,;\,s)}{p(\overline{l}_{k},\overline{a}_{k-1}\,;\,s)}\nonumber
    \\
    &>& 0.    \nonumber
  \end{eqnarray}
  Here \eqref{sil7} holds by \eqref{controlpar} and \eqref{sil9} holds
  by \eqref{es}.  The induction is established.  \qed\end{proof}

\begin{theorem}
  \label{thm:sidc}
  Suppose the conditions of \lemref{sidc} apply, and, further, that we
  have sequential irrelevance in the interventional regime $s$:
  \begin{equation}
    \label{eq:sis}
    \ind {L_i} {\overline U_{i-1}} {(\overline L_{i-1}, \overline A_{i-1}\,;\, s)} 
    \quad (i = 1, \ldots, n+1).
  \end{equation} 

  Then the simple stability property \eqref{st} holds.
\end{theorem}

\begin{proof}
  The result will be established if we can show that, for any $l_{i}$,
  we can find a function $w(\overline L_{i-1}, \overline A_{i-1})$
  such that, for both $\sigma = o$ and $\sigma = s$,
  \begin{equation*}
    p(l_{i} \mid \overline{l}_{i-1},\overline{a}_{i-1}\,;\,\sigma)=w(\overline{l}_{i-1},\overline{a}_{i-1})
  \end{equation*} 
  whenever $p(\overline{l}_{i-1},\overline{a}_{i-1}\,;\,\sigma)>0.$

  This is trivially possible if either regime gives probability $0$ to
  $(\overline{l}_{i-1},\overline{a}_{i-1})$.  So suppose
  $p(\overline{l}_{i-1},\overline{a}_{i-1}\,;\,\sigma)>0$ for both
  regimes.  Then
  \begin{equation}
    p(l_{i} \mid \overline{l}_{i-1},\overline{a}_{i-1}\,;\,o)
    =\sum_{\overline{u}_{i-1}}{}^{'}p(l_{i} \mid
    \overline{u}_{i-1}, \overline{l}_{i-1},\overline{a}_{i-1}\,;\,o)\,
    \times\, p(\overline{u}_{i-1} \mid
    \overline{l}_{i-1},\overline{a}_{i-1}\,;\,o)
    \label{eq:sit1}
  \end{equation}
  where $\sum'$ denotes summation restricted to terms for which
  $p(\overline{u}_{i-1},\overline{l}_{i-1},\overline{a}_{i-1}\,;\,o) >
  0$---and so, by \lemref{sidc},
  $p(\overline{u}_{i-1},\overline{l}_{i-1},\overline{a}_{i-1}\,;\,s)>0$.
  Then by \eqref{es},
  \begin{eqnarray}
    p(l_{i} \mid \overline{l}_{i-1},\overline{a}_{i-1}\,;\,o)
    &=&\sum_{\overline{u}_{i-1}}{}^{'}p(l_{i} \mid \overline{u}_{i-1},\overline{l}_{i-1},\overline{a}_{i-1}\,;\,s)\, \times\,p(\overline{u}_{i-1} \mid \overline{l}_{i-1},\overline{a}_{i-1}\,;\,o) \nonumber \\
    &=&\sum_{\overline{u}_{i-1}}{}^{'}p(l_{i} \mid \overline{l}_{i-1},\overline{a}_{i-1}\,;\,s)\,\times\, p(\overline{u}_{i-1} \mid \overline{l}_{i-1},\overline{a}_{i-1}\,;\,o) \label{eq:sit2}\\
    &=&p(l_{i} \mid \overline{l}_{i-1},\overline{a}_{i-1}\,;\,s)
    \nonumber
  \end{eqnarray}
  where \eqref{sit2} holds by \eqref{sis}.  Thus we can take
  \begin{equation*}
    w(\overline{l}_{i-1},\overline{a}_{i-1}):=p({l}_{i} \mid \overline{l}_{i-1},\overline{a}_{i-1}\,;\,s)
  \end{equation*} 
  to conclude the proof.  \qed\end{proof}

\Countref{discretesi} in the Appendix demonstrates that, even in this
discrete case, to deduce simple stability under the conditions of
\lemref{sidc} it is not sufficient to impose sequential irrelevance
only for the observational regime $o$.

\section{Conclusion}
\label{sec:summary}

The decision-theoretic approach to causal inference focuses on the
possibilities for transferring probabilistic information between
different stochastic regimes.  In this paper we have developed a
formal underpinning for this approach, based on an extension of the
axiomatic theory of conditional independence to include non-stochastic
variables.  This formal foundation now supplies a rigorous
justification for various more informal arguments that have previously
been presented \citep{Dawid1979,Dawid2002,dawiddidelez}.

In applying this theory to the problem of dynamic treatment
assignment, we have shown how, and under what conditions, the
assumptions of sequential randomization or sequential irrelevance can
support observational identification of the consequence of some
treatment strategy under consideration.  This is straightforward for
sequential randomization, but somewhat less so for sequential
irrelevance, where in general additional positivity conditions are
required; however, we have shown that these may be dispensed with when
all variables are discrete.


%
%

\pagebreak
\appendix

\section{Appendix: The need for positivity}


\begin{counterexample}
  \label{count:appb}
  The following counter-example illustrates what can go wrong when we
  do not have positivity: even when a property such as
  \eqref{nonstocci} holds, we can not use just any version of the
  conditional expectation in one regime to serve as a version of this
  conditional expectation in another regime.

  Consider a sequential decision problem of $n=2$ stages with domain
  variables $L_1$, $A$ and $L_2$, where $A$ is a binary variable with
  $A=0$ denoting no treatment and $A=1$ denoting treatment.  In the
  observational regime $o$, the treatment is never given:
  $\pr_{o}(A=0)=1$; while in the interventional regime $s$, the
  treatment is always given: $\pr_{s}(A=1)=1$.  We thus have failure
  of the positivity requirement of \defref{positivity}.

  Suppose that,in both regimes, $L_1= 0$ or $1$ each with probability
  $1/2$, and $L_2 = L_1 + A$.  Then, with $\sigma$ denoting the regime
  indicator taking values in $\cS = \{o, s\}$, we trivially have $\ind
  {L_{2}}{\sigma}{(L_{1}, A)}$.

  Now consider the variables
  \begin{equation*}
    W_{o}
    = \left\{
      \begin{array}{rl}
        L_1 & \text{if } A=0 \\
        0 & \text{if } A=1
      \end{array} \right.
  \end{equation*}
  and
  \begin{equation*}
    W_{s} = \left\{
      \begin{array}{rl}
        2 & \text{if } A=0 \\
        L_1+1 & \text{if } A=1.
      \end{array} \right.
  \end{equation*}
  
  Then $W_{o}=L_2\,\, \as{\pr_0}$, so $W_o$ serves as a version of
  $\E(L_2\cd L_1,A\,;\,o)$; also $W_{s}=L_2 \,\,\as{\pr_s}$, so $W_s$
  serves as a version of $\E(L_2\cd L_1,A\,;\,s)$.  However, almost
  surely under both $\pr_o$ and $\pr_s$, $W_{o} \neq W_{s}$, and
  neither of these variables supplies a version of $\E(L_2\cd L_1,A)$
  simultaneously valid in both regimes.  \qed
\end{counterexample}


\begin{counterexample}
  \label{count:discretesi}
  In \secref{dc} we have seen that, when all random variables are
  discrete and the conditions of \lemref{sidc} are satisfied, in order
  to be able to deduce simple stability it is sufficient to require
  sequential irrelevance only for the interventional regime.  However,
  without the positivity assumption simple stability does not follow
  if, additionally to the requirements of \lemref{sidc}, we instead
  require sequential irrelevance only for the observational regime.

  Consider a sequential decision problem of $n=2$ stages with extended
  information base ${\cal I}':= (U, A,Y)$.  The joint distributions of
  the variables in ${\cal I}'$ in the two regimes $\sigma= 0$ and
  $\sigma=s$ are supposed given by \tabref{nonlin}, where the
  probabilities are to be taken over $1500$ (\eg,
  $P(U=0,A=1,Y=0\,;\,s)= {252}/{1500}$).

  This problem does not exhibit extended positivity, since
  $P(U=1,A=1;\sigma=o)=0$; that is, in the observational regime, $U =
  1 \Rightarrow A = 0$.  Such a case might occur if, for example, $U$
  represents a patient's history of an allergic reaction to the
  treatment, which the doctor under observation knows about and takes
  into account, deciding that presence of the allergy should always
  preclude prescribing the treatment.  However the allergy information
  is not available to the decision-maker operating strategy $s$.

  The reader may check that extended stability, \eqref{es}, holds,
  \viz\ $\ind Y \sigma {U,A}$, and that $s$ is a control strategy:
  \eqref{controlpar} holds, \viz\ $\ind A U {\sigma=s}$.  Also,
  sequential irrelevance, \eqref{pa1a}, holds for the observational
  regime, \viz\ $\ind Y U {A; \sigma=s}$, though not the
  interventional regime, since $\nind Y U {A=1;\sigma=s}$.  And now
  simple stability, \eqref{st}, does not hold, since $\nind Y \sigma
  {A=1}$.  \qed

  \begin{table}[h]
    \centering 
    \begin{tabular}{c |rr } 
      & \multicolumn{1}{c}{$\sigma=o$} & \multicolumn{1}{c}{$\sigma=s$} \\ [0.5ex] 
      \hline 
      $P(U=0,A=0,Y=0)$ & 135 & 180 \\ 
      $P(U=0,A=0,Y=1)$ & 240 & 320\\
      $P(U=0,A=1,Y=0)$ & 50 & 25 \\ 
      $P(U=0,A=1,Y=1)$ & 200 & 100 \\ 
      $P(U=1,A=0,Y=0)$ & 315 & 252 \\
      $P(U=1,A=0,Y=1)$ & 560 & 448 \\
      $P(U=1,A=1,Y=0)$ & 0 & 98 \\
      $P(U=1,A=1,Y=1)$ & 0 & 77\\
      \hline 
    \end{tabular}
    \caption{Sequential irrelevance in the observational regime} 
    \label{tab:nonlin} 
  \end{table}
\end{counterexample}
\end{document}